\documentclass[12pt]{article}
\usepackage{amsmath,amssymb,amsfonts,amsthm}
\usepackage[margin=1in]{geometry}
\usepackage{subcaption,graphicx}

\theoremstyle{plain}

\newcommand{\cn}{\tilde{g}}

\newtheorem{theorem}{Theorem}[section]
\newtheorem{proposition}[theorem]{Proposition}
\newtheorem{observation}[theorem]{Observation}
\newtheorem{lemma}[theorem]{Lemma}

\title{Settling the nonorientable genus of the nearly complete bipartite graphs}
\author{Warren Singh and Timothy Sun\\Department of Computer Science\\San Francisco State University}
\date{}

\begin{document}

\maketitle

\begin{abstract}
A graph is said to be nearly complete bipartite if it can be obtained by deleting a set of independent edges from a complete bipartite graph. The nonorientable genus of such graphs is known except in a few cases where the sizes of the partite classes differ by at most one, and a maximum matching is deleted. We resolve these missing cases using three classic tools for constructing genus embeddings of the complete bipartite graphs: current graphs, diamond sums, and the direct rotation systems of Ringel. 
\end{abstract}

\section{Introduction}

Ringel \cite{Ringel-Orientable, Ringel-Nonorientable} gave the first proof of the orientable and nonorientable genus formulas for complete bipartite graphs by explicitly writing down rotation systems and face sets for each such graph. Other approaches have been discovered in attempts to simplify Ringel's proof: Gross and Tucker's monograph \cite{GrossTucker} presented a \emph{current graph} construction, and Bouchet \cite{Bouchet-Diamond} gave a proof by induction. Bouchet's main tool, the \emph{diamond sum} operation on embeddings, has received recent attention, playing a central role in proofs of other genus formulas, e.g., the nonorientable genus of the complete tripartite graphs \cite{Ellingham-Tripartite}, the Even Map Color Theorem \cite{EvenMapColor}, and minimal quadrangulations \cite{Abusaif-MinQuad}. 

Mohar, Parsons, and Pisanski \cite{Mohar-Orientable} extended Bouchet's proof to calculate the orientable genus of the \emph{nearly complete bipartite graphs}. In their notation, the graph $G(m,n,k)$ stands for the complete bipartite graph $K_{m,n}$ with $k$ independent edges deleted. Mohar \cite{Mohar-Nonorientable} later proved analogues for the nonorientable case using similar ideas, but there were two families that were left unsolved. Our goal is to resolve these gaps:

\begin{theorem}
For all $n \geq 3$, $G(2n+1, 2n, 2n)$ has a nonorientable quadrangular embedding.
\label{thm-imbalanced}
\end{theorem}
\begin{theorem}
For all $n \geq 6$, the graph $G(n,n,n)$ has a nonorientable quadrangular embedding.
\label{thm-balanced}
\end{theorem}

Mohar \emph{et al.}\ \cite{Mohar-Orientable, Mohar-Nonorientable} reduced the problem to determining which such graphs have \emph{quadrangular embeddings}, embeddings where every face is a quadrilateral. The genus of the remaining graphs, which are either subgraphs of these graphs or the complete bipartite graphs, follows from a lower bound derived from Euler's formula. For example, $K_{4,4}$ has a quadrangular embedding in the torus, which means that $G(4,4,3)$ also embeds in the torus. $G(4,4,3)$ has too many edges to be planar, so its genus must be 1. 

The diamond sum techniques used to solve the orientable problem were unable to cover all of the additional cases introduced in the nonorientable analogue, namely the graphs which are expected to quadrangulate surfaces with odd Euler characteristic. Furthermore, some of the quadrangular embeddings cannot be the diamond sum of two smaller nearly complete bipartite graphs. Mohar \emph{et al.}\ \cite{Mohar-Orientable} relied on Jungerman, Stahl, and White \cite{Jungerman-Hypergraph} for orientable quadrangular embeddings of $G(n,n,n)$ when $n \equiv 0,1 \pmod{4}$. The latter paper did not consider the nonorientable genus outside of stating a lower bound.

Our constructions apply all three aforementioned methods: current graphs, diamond sums, and Ringel's rotation systems. To prove Theorem \ref{thm-balanced}, we combine a result of Mohar \cite{Mohar-Nonorientable} and two families of nonorientable current graphs in the style of Jungerman \emph{et al.}\ \cite{Jungerman-Hypergraph}. Our proof of Theorem \ref{thm-imbalanced} requires a careful sequence of diamond sums, where it is necessary to know the exact rotations at vertices. We found it easiest to use Ringel's embeddings \cite{Ringel-Nonorientable} here, since he explicitly provided the rotation systems. 

Recently, Theorems \ref{thm-imbalanced} and \ref{thm-balanced} were also proven independently by Lv\ \cite{Lv}. The methods employed in that work bear little similarity to what we present here. 

\section{Embeddings of bipartite graphs}

For background on topological graph theory, especially the theory of current graphs, see Gross and Tucker \cite{GrossTucker}. 

The \emph{nearly complete bipartite graph} $G(m,n,k)$ is the complete bipartite graph $K_{m,n}$ with $k$ independent edges deleted. There are $m$ \emph{white} vertices and $n$ \emph{black} vertices. A vertex is said to be \emph{saturated} if it is adjacent to every vertex of the other color, and \emph{unsaturated}, otherwise.

If $G$ is a graph and $S$ is a surface, an embedding $\phi\colon G \to S$ is said to be \emph{cellular} if $S \setminus \phi(G)$ is a disjoint union of disks, which we call \emph{faces}. Let $N_h$ denote the surface formed by taking the sphere and adding $h \geq 0$ crosscaps. Then, Euler's formula for nonorientable surfaces states that a cellular embedding $\phi\colon G \to N_h$ satisfies
$$|V(G)|-|E(G)|+|F(G, \phi)| = 2 -h,$$
where $F(G,\phi)$ denotes the set of faces of $\phi$. The \emph{nonorientable genus} $\cn(G)$ is the smallest nonnegative integer $h$ such that $G$ has an embedding in $N_h$. For convenience, we allow $\cn(G) = 0$ to avoid making exceptions for planar graphs. For this reason, some authors prefer to call this parameter the \emph{crosscap number} instead. 

If an embedded graph is bipartite and is not $K_2$, then the length of any face is at least 4, and hence $4|F| \leq 2|E|$. Substituting this into Euler's formula yields the lower bound
$$\cn(G) \geq \frac{|E|-2|V|+4}{2},$$
with equality if and only if $G$ has a quadrangular embedding in that surface. For the family of graphs we consider, this lower bound reads
$$\cn(G(m,n,k)) \geq \frac{(m-2)(n-2)-k}{2}.$$
The proofs of Theorems \ref{thm-imbalanced} and \ref{thm-balanced} complete the nonorientable genus formula for the nearly complete bipartite graphs:

\begin{theorem}
Let $m, n, k$ be nonnegative integers such that $m, n \geq 3$ and $k \leq \min(m,n)$. Then,
$$\cn(G(m,n,k)) \geq \left\lceil\frac{(m-2)(n-2)-k}{2}\right\rceil,$$
except for $\cn(G(3,3,3)) = 0$, $\cn(G(5,4,4)) = 2$, and $\cn(G(5,5,5)) = 3$. 
\end{theorem}

The first of these exceptions is isomorphic to a cycle of length 6, and has so few edges that the formula results in a negative number. Mohar \cite{Mohar-Nonorientable} showed that the latter two graphs do not embed in $N_1$ and $N_2$, respectively. 

Given an undirected graph $G$, each edge $e \in E(G)$ induces two directed arcs $e^+$ and $e^-$ pointing in opposite directions. We write $E(G)^+$ to denote the set of all such arcs. A \emph{rotation} of a vertex $v$ is a cyclic permutation of the arcs leaving $v$. When $G$ is simple, it suffices to give a cyclic permutation of the neighbors of $v$. A \emph{(general) rotation system} consists of a rotation for each vertex and an \emph{edge signature} $\lambda\colon E(G) \to \{-1,1\}$, and describes a cellular embedding of the graph on a possibly nonorientable surface. We say that an edge is \emph{twisted} if its signature is $-1$. The set of faces can be traced from the rotation system (see Section 3.2.6 of \cite{GrossTucker}), where traversing a twisted edge reverses the local orientation of a face boundary walk. Different general rotation systems can describe the same embedding: a \emph{vertex flip} reverses the rotation of some vertex and switches the signature of all of its incident edges. A rotation system with twisted edges can still be orientable, unless there is a closed walk that traverses an odd number of twisted edges.

\section{Diamond sum constructions}

Bouchet \cite{Bouchet-Diamond} introduced the \emph{diamond sum} operation for building embeddings of complete bipartite graphs out of smaller embeddings. Later authors \cite{Mohar-Orientable, Mohar-Nonorientable, Ellingham-Tripartite, EvenMapColor} interpreted Bouchet's operation in primal form and generalized it to other families of graphs.

Let $\phi\colon G \to S$ and $\phi'\colon G' \to S'$ be two cellular embeddings, and let $v \in V(G)$ and $v' \in V(G')$ be two vertices of the same degree $d$. As seen in Figure \ref{fig-diamond}, a \emph{diamond sum} of $\phi$ and $\phi'$ at $v$ and $v'$ combines the two embeddings in the following way: let $D$ (resp. $D'$) be a closed disk in $S$ (resp. $S'$) that contains $v$ (resp. $v'$) and its incident edges, such that its neighbors are on the boundary of the disk and no other part of the graph intersects the disk. Then, delete the interiors of $D$ and $D'$ and glue the two embeddings at the resulting boundaries, making sure that each vertex on the boundary of $D$ is identified with a vertex on the boundary of $D'$. There are different ways of forming a diamond sum, depending on which pairs of vertices are identified. 

\begin{figure}[t]
\centering
    \begin{subfigure}[b]{0.99\textwidth}
    \centering
        \includegraphics[scale=1]{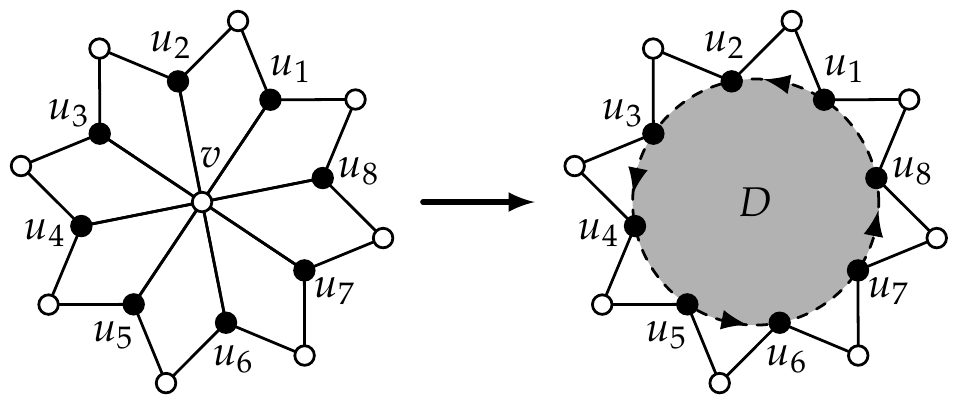}
        \caption{}
        \label{subfig-o34}
    \end{subfigure}
    \begin{subfigure}[b]{0.99\textwidth}
    \centering
        \includegraphics[scale=1]{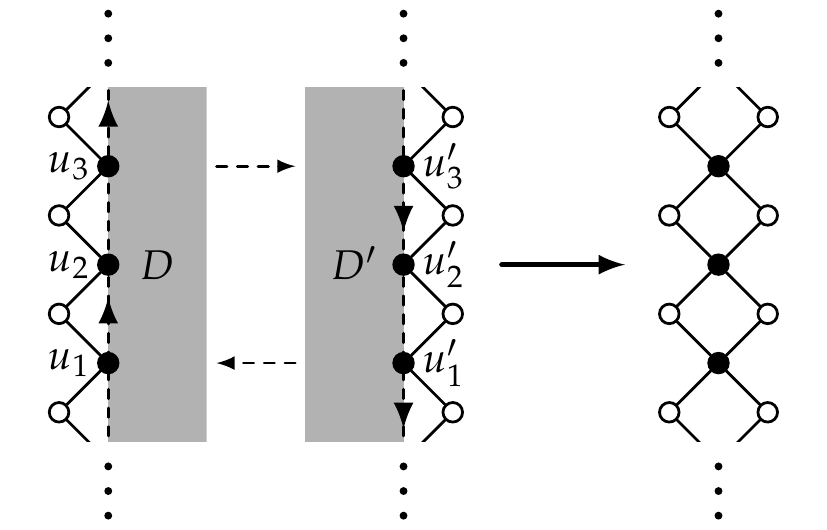}
        \caption{}
        \label{subfig-o40}
    \end{subfigure}
\caption{A diamond sum and the orientation-preserving vertex merging scheme.}
\label{fig-diamond}
\end{figure}

Combinatorially, the diamond sum operation on rotation systems requires a little care. In order to avoid making modifications to the local orientations at vertices, we glue the disks $D$ and $D'$ in the following ``anti-parallel'' manner. Suppose the rotation of $v$ is $(u_1 \,\, u_2 \,\, u_3 \, \dotsc \, u_d)$ and the rotation of $v'$ is $(u_d' \,\, u_{d-1}' \,\, u_{d-2}' \, \dotsc \, u_1')$. Then one possible diamond sum merges $u_i$ with $u_i'$, for each $i = 1, \dotsc, d$, and all others are found by applying a cyclic shift to the vertices on, say, disk $D'$. Furthermore, excising a disk around a vertex $v$ requires that all edges incident with $v$ have signature $1$ in the rotation system. Any rotation system can be modified to satisfy this property by flipping the appropriate neighbors of $v$. 

In this paper, we only consider diamond sums of quadrangular embeddings of nearly complete bipartite graphs, where at least one of the two embeddings is nonorientable. Then, the resulting embedding is also quadrangular and nonorientable. Furthermore, we always excise white vertices and merge black vertices. Under these restrictions, we make an observation about how the diamond sum operation effectively leaves much of the embeddings unchanged:

\begin{observation}
The diamond sum operation preserves the rotations of all non-excised white vertices, up to reversal.
\label{obs-diamond}
\end{observation}

When taking a diamond sum of quadrangular embeddings of $G(m_1, n, k_1)$ and $G(m_2, n, k_2)$, the underlying graph is another nearly complete bipartite graph $G(m_1+m_2-2, n, k_1+k_2)$ if two conditions are satisfied:
\begin{itemize}
\item the two deleted white vertices are saturated, and
\item there is a choice of gluing where no unsaturated black vertices are merged together.
\end{itemize}
These conditions are visualized in Figure \ref{fig-dsum}, where a diamond sum of some embeddings of $G(4,6,2)$ and $G(5,6,2)$ results in one of $G(7,6,4)$. 

\begin{figure}[]
\centering
\includegraphics[scale=1]{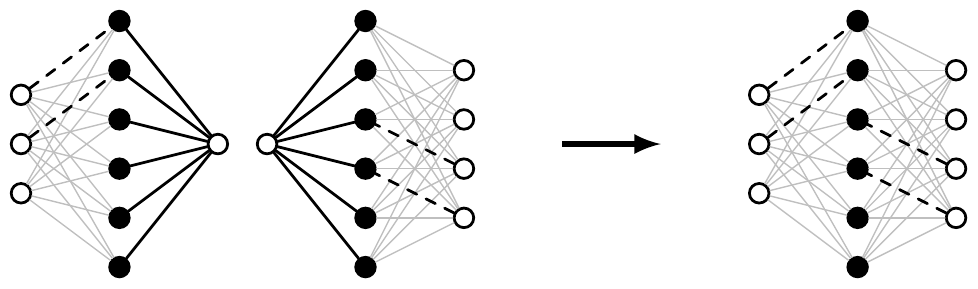}
\caption{A schematic of a diamond sum operation.}
\label{fig-dsum}
\end{figure}

The latter condition suggests that one might need information about the actual rotation system to ensure a valid gluing. Mohar \emph{et al.}\ \cite{Mohar-Orientable, Mohar-Nonorientable} circumvented this problem by only considering diamond sums where, say, $k_1 < n$ and $k_2 \leq 1$. In this situation, one can choose a gluing where the one unsaturated black vertex in the second graph is merged with any saturated black vertex in the first graph. When both graphs have many unsaturated black vertices, then such a choice might not exist. 

The graph $G(2n+1, 2n, 2n)$ can be formed by starting with $K_{n+1,2n} = G(n+1, 2n, 0)$ and repeatedly taking diamond sums with $G(3, 2n, 2)$. We first explore the space of embeddings of the latter graph.

\begin{lemma}
For any $n \geq 2$ and odd $p < 2n-2$, there exists a quadrangular embedding of $G(3, 2n, 2)$ where the two unsaturated black vertices are separated by $p$ vertices in the rotation of the one saturated white vertex. 
\label{lem-k3n}
\end{lemma}
\begin{proof}
Consider any quadrangular embedding of $K_{3, 2n-2} = G(3, 2n-2, 0)$ and let $a, b, c$ be the names of the white vertices. Because the graph is simple, the other white vertex in the quadrilateral faces incident with $a$ alternates between $b$ and $c$ as we go around the rotation of $a$, as seen in Figure \ref{fig-g3n2}. Additional degree 2 black vertices can be inserted into any one of these quadrilateral faces. We add two such vertices, one incident with $b$, and the other incident with $c$ to form a quadrangular embedding of the graph $G(3, 2n, 2)$. Vertex $a$ is saturated, and by an appropriate choice of faces, the two new vertices can be separated by any odd number of vertices in the rotation of $a$. 
\end{proof}

\begin{figure}[]
\centering
\includegraphics[scale=1]{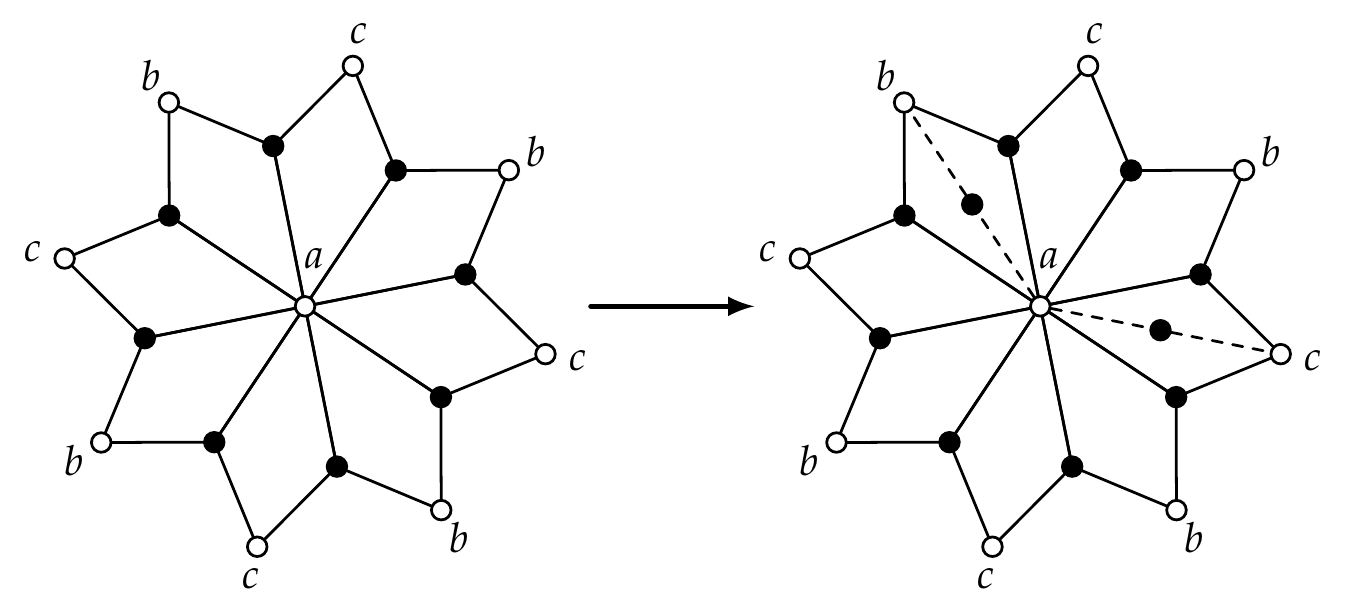}
\caption{Constructing $G(3,10,2)$ from $G(3,8,0)$ with $p = 3$ (or 5).}
\label{fig-g3n2}
\end{figure}

We say that a quadrangular embedding of $K_{m,2n}$, $m \geq n$, has an \emph{odd pairing} if there is a partitioning of the $2n$ black vertices into pairs, such that each pair can be assigned a distinct white vertex satisfying the following property: if the pair $\{i,j\}$ is mapped to vertex $\alpha$, then $i$ and $j$ are separated by an odd number of vertices in the rotation of $\alpha$. The intuition behind this definition is to determine which white vertices can be used to ``desaturate'' a specific pair of black vertices using an embedding constructed in Lemma \ref{lem-k3n}. But once a white vertex is excised by a diamond sum, it cannot be used again. 

\begin{lemma}
If some nonorientable quadrangular embedding of $K_{m,2n}$, $m \geq n$, has an odd pairing, then there exists a nonorientable quadrangular embedding of $G(m+n, 2n, 2n)$.
\label{lem-odd}
\end{lemma}
\begin{proof}
We define a sequence of embeddings $\phi_0, \phi_1, \dotsc, \phi_n$, where $\phi_0$ is the initial embedding of $K_{m,2n}$. Suppose the odd pairing maps $\{i_k, j_k\} \mapsto \alpha_k$, for $k = 1, \dotsc, n$. For each $\phi_{k-1}$, consider the vertices $i_k$, $j_k$, and $\alpha_k$. Following Lemma \ref{lem-k3n}, there is a quadrangular embedding of $G(3,2n,2)$ where the two unsaturated black vertices, call them $i$ and $j$, are separated by the same number of vertices as $i_k$ and $j_k$ in the rotation of $\alpha_k$. By Observation \ref{obs-diamond}, $i_k$ and $j_k$ are still the same distance apart in the rotation of $\alpha_k$ in $\phi_{k-1}$ as in the original embedding $\phi_0$. To obtain $\phi_k$, we take the diamond sum of $\phi_{k-1}$ and this embedding of $G(3,2n,2)$ at vertices $\alpha_k$ and $a$, where $i$ and $j$ are glued to $i_k$ and $j_k$ (in either order). Each embedding $\phi_k$ is nonorientable, quadrangular, and of the graph $G(m+k,2n,2k)$. Thus, $\phi_n$ is the desired embedding. 
\end{proof}

Ringel \cite{Ringel-Nonorientable} gave the first construction for nonorientable quadrangular embeddings of $K_{m,n}$, but the rotation systems are missing edge signatures (instead, Ringel writes down all of the faces explicitly). Fortunately, the edge signature is irrelevant in the definition of an odd pairing and the construction in Lemma \ref{lem-odd}: the property of being separated by an odd number of vertices in some rotation still holds after any number of vertex flips. 

\begin{proof}[Proof of Theorem \ref{thm-imbalanced}]
We find an odd pairing in each of Ringel's \cite{Ringel-Nonorientable} rotation systems of $K_{n+1, 2n}$, $n \geq 3$. Label the black vertices $1, 2, \dotsc, 2n$. The rotation of all but two of the white vertices is simply 
$$\begin{array}{rlllllllllllllllllllllllllllll}
(1 & 2 & 3 & 4 & \dotsc & 2n{-}1 & 2n).
\end{array}$$
In an odd pairing, any pair of black vertices of the same parity can be assigned to such a vertex. 

When $n$ is even, the rotations of the other two vertices, call them $a$ and $b$, are 
$$\begin{array}{rlllllllllllllllllllllllllllll}
a. & (1 & 2 & 2n & 2n{-}1 & 3 & 4 & 2n{-}2 & 2n{-}3 & \dotsc & n{+}2 & n{+}1) \\
b. & (1 & 2n & 2 & 3 & 2n{-}1 & 2n{-}2 & 4 & 5 & \dotsc & n & n{+}1) \\
\end{array}$$
Pair the black vertices as follows: 
$$\{1, 3\}, \{2, 4\}, \{5, 7\}, \{6, 8\}, \dotsc, \{2n{-}3, 2n{-}1\}, \{2n{-}2, 2n\}.$$
Assign the pair $\{1, 3\}$ to vertex $a$, and assign all the other pairs arbitrarily to any of the remaining vertices except $b$. 

When $n$ is odd, the rotations of vertices $a$ and $b$ are both
$$\begin{array}{rlllllllllllllllllllllllllllll}
(1 & 3 & 2 & 4 & 5 & 6 & \dotsc & 2n{-}1 & 2n)
\end{array}$$
where only $2$ and $3$ are swapped. For this case, our pairs are 
$$\{1, 2\}, \{3, 5\}, \{4, 6\}, \{7, 9\}, \{8, 10\}, \dotsc, \{2n{-}3, 2n{-}1\}, \{2n{-}2, 2n\}.$$
Like before, assign the first pair to vertex $a$, and all others to the vertices besides $b$. 

By Lemma \ref{lem-odd}, these odd pairings produce nonorientable quadrangular embeddings of $G(2n+1, 2n, 2n)$. 
\end{proof}

If we started with $K_{n,2n}$ instead of $K_{n+1,2n}$, we would obtain quadrangular embeddings of $G(2n,2n,2n)$. For $n \geq 4$, Mohar \cite{Mohar-Nonorientable} already outlined an approach to finding such embeddings (see Lemma \ref{lem-mohar} below). The $n = 2$ case is the cube graph, which is planar, so the $n = 3$ case remains to be solved. 

\begin{proposition}
$G(6, 6, 6)$ has a nonorientable quadrangular embedding. 
\label{prop-g666}
\end{proposition}
\begin{proof}
The rotations of the white vertices in Ringel's embedding of $K_{3,6}$ are:
$$\begin{array}{rlllllllllllllllllllllllllllll}
a. & (1 & 2 & 6 & 5 & 3 & 4) \\
b. & (1 & 6 & 2 & 3 & 5 & 4) \\
c. & (1 & 2 & 3 & 4 & 5 & 6).
\end{array}$$
We use Lemma \ref{lem-odd} on the odd pairing $\{1, 3\} \mapsto a, \{2,5\} \mapsto b, \{4,6\} \mapsto c$. 
\end{proof}

\section{Current graph constructions}

We assume that the reader is familiar with current graph constructions. For more details, see Sections 4.4 and 6.1.3 of Gross and Tucker \cite{GrossTucker}. Our constructions are based on the family of current graphs in Jungerman \emph{et al.}\ \cite{Jungerman-Hypergraph}.

A \emph{current graph} $(\phi, \alpha)$ consists of a graph embedding $\phi\colon G \to S$ and an arc-labeling $\alpha\colon E(G)^+ \to \Gamma$ with elements from a group $\Gamma$. If $\lambda\colon E(G) \to \{-1,1\}$ is the edge signature, the arc-labeling $\alpha$ satisfies $\alpha(e^+) = -\lambda(e)\alpha(e^-)$ for every edge $e \in E(G)$. Given a face boundary walk $(e^\pm_1, e^\pm_2, e^\pm_3, \dotsc)$ of the embedding $\phi$, its \emph{log} replaces each $e^\pm_i$ with $\alpha(e^\pm_i)$ if the walk is currently in its original orientation or $-\alpha(e^\pm_i)$, otherwise.

When $n$ is odd, the nearly complete bipartite graphs $G(n,n,n)$ can be expressed as a Cayley graph on $\mathbb{Z}_{2n}$, where the generating set consists of the odd numbers $1, 3, \dotsc, n-2$. Then, the white and black vertices are, say, the even and odd vertices, respectively. We describe ``index 2'' current graphs that satisfy a standard set of properties:
\begin{itemize}
\item[(C1)] The embedding has two faces whose boundary walks are labeled $[0]$ and $[1]$. 
\item[(C2)] Every vertex has degree 4 and satisfies Kirchhoff's current law.
\item[(C3)] In each log, each element in the list $\pm 1, \pm 3, \dotsc, \pm(n-2) \in \mathbb{Z}_{2n}$ appears exactly once, and no other elements appear. 
\item[(C4)] Each edge is incident with both faces.
\end{itemize}
When these properties are satisfied, the \emph{derived embedding} of the current graph, which has vertex set $\mathbb{Z}_{2n}$ and is generated from the logs of the face boundary walks, is a quadrangular embedding of $G(n,n,n)$. Figures \ref{fig-g7}, \ref{fig-g-case1}, and \ref{fig-g-case3} describe such current graphs for each odd integer $n \geq 7$. The rotations of the vertices in the current graphs' embeddings match the way the current graphs are being drawn in the plane, and the orientation of each face boundary walk is indicated by a small fragment where the walk is in normal, unreversed behavior. 

\begin{figure}[]
\centering
\includegraphics[scale=1]{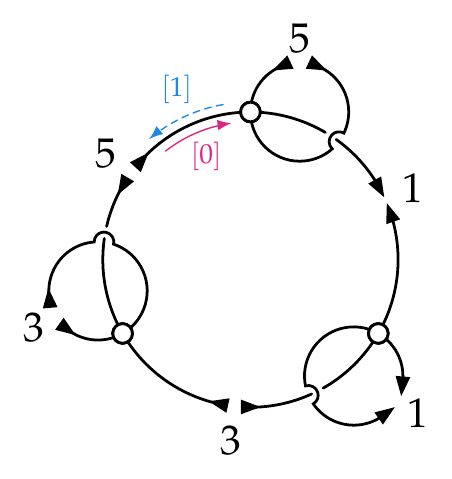}
\caption{Current graph with group $\mathbb{Z}_{14}$.}
\label{fig-g7}
\end{figure}

\begin{figure}[]
\centering
\includegraphics[scale=1]{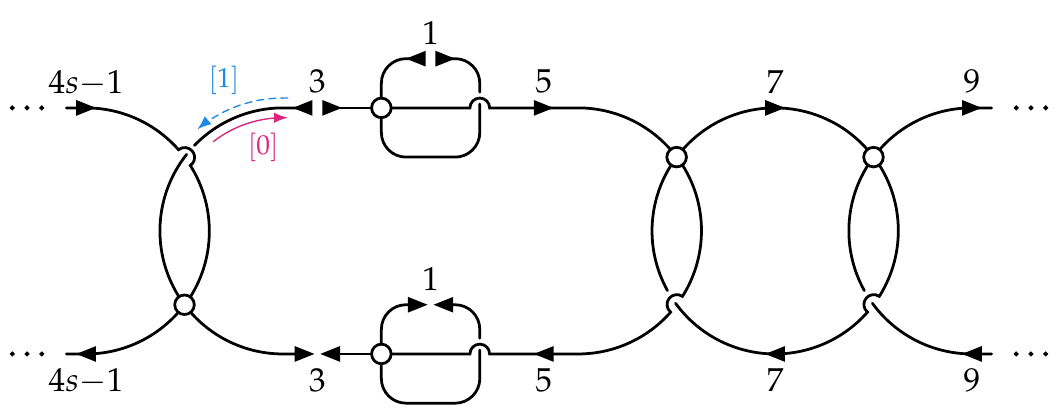}
\caption{Current graph with group $\mathbb{Z}_{8s+2}$, $s \geq 2$.}
\label{fig-g-case1}
\end{figure}

\begin{figure}[]
\centering
\includegraphics[scale=1]{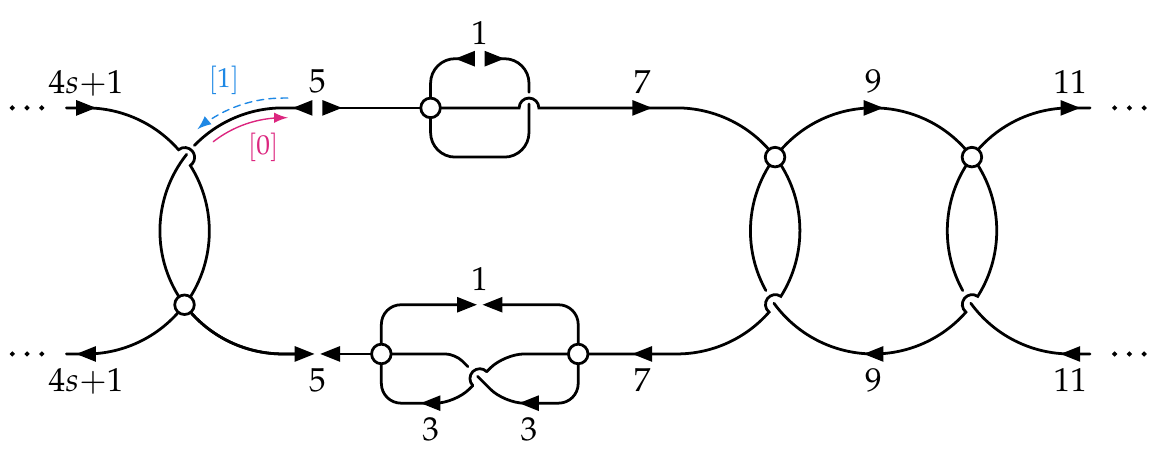}
\caption{Current graph with group $\mathbb{Z}_{8s+6}$, $s \geq 2$.}
\label{fig-g-case3}
\end{figure}

Jungerman \emph{et al.}\ \cite{Jungerman-Hypergraph} announced, but did not prove, the existence of orientable quadrangular embeddings of $G(4s, 4s, 4s)$. A small modification to their family of current graphs would suffice, but we follow a different construction due to Mohar \cite{Mohar-Nonorientable}. With a more complicated variant of the diamond sum operation, Mohar used the toroidal embedding of $G(5,5,5)$ to construct orientable and nonorientable quadrangular embeddings of $G(4s, 4s, 4s)$, for all $k \geq 2$. He observed that replacing ``5'' with ``7'' would yield most of a different residue:

\begin{lemma}[Mohar \cite{Mohar-Nonorientable}, Proposition 1]
If $G(7,7,7)$ has a quadrangular embedding, then $G(4s+2, 4s+2, 4s+2)$ has a nonorientable quadrangular embedding, for all $s \geq 2$. 
\label{lem-mohar}
\end{lemma}

\begin{proof}[Proof of Theorem \ref{thm-balanced}]
Mohar \cite{Mohar-Nonorientable} proved the $n \equiv 0 \pmod{4}$ case. The current graphs in Figures \ref{fig-g7}, \ref{fig-g-case1}, and \ref{fig-g-case3} generate quadrangular embeddings of $G(n,n,n)$ when $n$ is odd and $n \geq 7$. The derived embeddings from Figures \ref{fig-g7} and \ref{fig-g-case3} must be nonorientable, simply because $G(4s+3,4s+3,4s+3)$ cannot quadrangulate an orientable surface. We check that the derived embeddings of Figure \ref{fig-g-case1} are also nonorientable. 

The twisted edges in a derived embedding correspond to the edges in the current graph traversed twice in the same direction by face boundary walks. For the face orientations in Figure \ref{fig-g-case1}, those edges are just the self-loops, and hence the only twisted edges are of the form $(i,i+1)$. The closed walk $0 \to 1 \to 6 \to 3 \to 0$ in each derived embedding traverses exactly one twisted edge, so the embedding is nonorientable. 

By Proposition \ref{prop-g666} and Lemma \ref{lem-mohar}, there also exist nonorientable quadrangular embeddings of $G(4s+2,4s+2,4s+2)$ for all $s \geq 1$. 
\end{proof}

\bibliographystyle{alpha}
\bibliography{biblio}

\end{document}